\documentclass[12pt]{amsart}

\title{Dynamically distinguishing polynomials}
\author{Andrew Bridy}
\address{Department of Mathematics, Texas A\&M University}
\email{\href{mailto:andrewbridy@math.tamu.edu}{andrewbridy@math.tamu.edu}}
\author{Derek Garton}
\address{Fariborz Maseeh Department of Mathematics and Statistics, Portland State University}
\email{\href{mailto:gartondw@pdx.edu}{gartondw@pdx.edu}}
\date{\today}



\usepackage{amssymb,url,MnSymbol}

\usepackage{hyperref}
\hypersetup{breaklinks=true,
colorlinks=true,
urlcolor=blue,
linkcolor=cyan
}

\usepackage[top=1in,bottom=1in,left=1in,right=1in]{geometry}

\usepackage[capitalise,nameinlink]{cleveref}

\makeatletter
\def\imod#1{\allowbreak\mkern8mu{\operator@font mod}\,\,#1}
\makeatother

\newcommand{\C}{\ensuremath{{\mathbb{C}}}}
\newcommand{\Z}{\ensuremath{{\mathbb{Z}}}}

\renewcommand{\P}{\ensuremath{{\mathbb{P}}}}
\newcommand{\Q}{\ensuremath{{\mathbb{Q}}}}

\newcommand{\F}{\ensuremath{{\mathbb{F}}}}

\newcommand{\lv}{\ensuremath{\left\vert}}
\newcommand{\rv}{\ensuremath{\right\vert}}
\newcommand{\lp}{\ensuremath{\left(}}
\newcommand{\rp}{\ensuremath{\right)}}
\newcommand{\lb}{\ensuremath{\left\{}}
\newcommand{\rb}{\ensuremath{\right\}}}
\newcommand{\lc}{\ensuremath{\left[}}
\newcommand{\rc}{\ensuremath{\right]}}

\DeclareMathOperator{\Gal}{Gal}
\DeclareMathOperator{\Aut}{Aut}
\DeclareMathOperator*{\Fix}{Fix}
\DeclareMathOperator*{\Char}{char}
\DeclareMathOperator*{\ord}{ord}

\DeclareMathOperator*{\Disc}{Disc}

\theoremstyle{plain}
\newtheorem{final}{Theorem}[section]
\newtheorem{fixedpointsinwreaths}[final]{Theorem}
\newtheorem{pnbound}[final]{Theorem}
\newtheorem{comparativerecurrence}[final]{Lemma}
\newtheorem{recurrence}[final]{Corollary}

\newtheorem{genmorton}[final]{Theorem}
\newtheorem{QbarQ}[final]{Corollary}
\newtheorem{automorphismlemma}[final]{Lemma}
\newtheorem{rencontresprop}[final]{Lemma}
\newtheorem{rootsandcycles}[final]{Corollary}

\newtheorem*{namedthm}{\namedthmname}
\newcounter{namedthm}

\theoremstyle{remark}
\newtheorem{fixedpoints}[final]{Example}
\newtheorem{GaloisWreathActions}[final]{Remark}
\newtheorem{HITremark}[final]{Remark}
\newtheorem{FDTremark}[final]{Remark}

\theoremstyle{definition}
\newtheorem{rkn}[final]{Definition}
\newtheorem{Pkn}[final]{Definition}

\makeatletter
\newenvironment{named}[1]
  {\def\namedthmname{#1}%
   \refstepcounter{namedthm}%
   \namedthm\def\@currentlabel{#1}}
  {\endnamedthm}
\makeatother

\begin{document}

\begin{abstract}

A polynomial with integer coefficients yields a family of dynamical systems indexed by primes as follows: for any prime $p$, reduce its coefficients mod $p$ and consider its action on the field $\F_p$. We say a subset of $\Z[x]$ is dynamically distinguishable mod $p$ if the associated mod $p$ dynamical systems are pairwise non-isomorphic. For any $k,M\in\Z_{>1}$, we prove that there are infinitely many sets of integers $\mathcal{M}$ of size $M$ such that $\lb x^k+m\mid m\in\mathcal{M}\rb$ is dynamically distinguishable mod $p$ for most $p$ (in the sense of natural density). Our proof uses the Galois theory of dynatomic polynomials largely developed by Morton, who proved that the Galois groups of these polynomials are often isomorphic to a particular family of wreath products. In the course of proving our result, we generalize Morton's work and compute statistics of these wreath products.
\end{abstract}

\maketitle

\section{Introduction}
\label{intro}

A \emph{\textup{(}discrete\textup{)} dynamical system} is a pair $\lp S,f\rp$ consisting of a set $S$ and a function $f:S\to S$.
The \emph{functional graph} of $(S,f)$, which we will denote by $\Gamma(S,f)$, is the directed graph whose set of vertices is $S$ and whose edges are given by the relation $s\to t$ if and only if $f(s)=t$. 

Recently there has been interest in the following problem: given a set $S$ and a family $\mathcal{F}$ of self-maps of $S$, describe or enumerate the set $M(S,\mathcal{F}):=\{\Gamma(S,f)\mid f\in\mathcal{F}\}/\simeq$, where for two directed graphs $\Gamma$ and $\Delta$, we write $\Gamma\simeq\Delta$ if they are isomorphic as directed graphs.
For example, for any $n\in\Z_{>0}$ and prime power $q$, Bach and the first author~\cite{BachBridy} bound the size of $M(S,\mathcal{F})$, where $S=\lp\F_q\rp^n$ and $\mathcal{F}$ is the set of affine-linear transformations from $S$ to itself. Konyagin et.\ al.~\cite{Ketal} give nontrivial upper and lower bounds on $M\lp\F_q,\lb f\in \F_q[x]\mid\deg(f)=d\rb\rp$. Similarly, Ostafe and Sha \cite{OstafeSha} give bounds on $M\lp\F_q,\mathcal{F}\rp$ for certain families $\mathcal{F}$ of rational functions and ``sparse'' polynomials. A special case of Theorem~2.8 of~\cite{Ketal} proves that
\[
\lv M\lp\F_q,\lb x^2+\alpha\mid\alpha\in\F_q\rb\rp\rv>q^{\frac{1}{4}+o(1)}
\]
as $q$ increases amongst odd prime powers.
Moreover, the authors suggest that it is ``most likely'' that for any rational prime $p$ with $p\nin\{2,17\}$,
\[
\lv M\lp\F_p,\lb x^2+\alpha\mid\alpha\in\F_p\rb\rp\rv=p.
\]
However, they also state that ``proving [this suggestion] may be difficult\ldots~as there is no intrinsic reason for this to be true.''

In this paper, we study the suggestion of ~\cite{Ketal} ``in reverse''; that is, we fix (integer polynomial) maps, then vary the set upon which they act by reducing these polynomials modulo rational primes.
Before stating our results, we introduce a bit of notation.
Denote the set of rational primes by $\mathcal{P}$.
For $f\in\Z[x]$ and $p\in\mathcal{P}$, write
\begin{itemize}
\item
$\lc f\rc_p$ for the polynomial in $\F_p[x]$ obtained by reducing the coefficients of $f$ mod $p$ and
\item
$\Gamma_{f,p}$ for $\Gamma\lp\F_p,[f]_p\rp$.
\end{itemize}
We say that a set $\mathcal{F}\subseteq\Z[x]$ is \emph{dynamically distinguishable} mod $p$ if $\Gamma_{f,p}\nsimeq\Gamma_{g,p}$ for all $f,g\in\mathcal{F}$ with $f\neq g$.
Let $\mu$ be the natural density on $\mathcal{P}$; that is, for any subset $P\subseteq\mathcal{P}$,
\[
\mu\lp P\rp:=\lim_{X\to\infty}
{\frac{\lv\lb p\in\mathcal{P}
\mid p\leq X\text{ and }p\in P\rb\rv}
{\lv\lb p\in\mathcal{P}\mid p\leq X\rb\rv}}
\hspace{20px}\text{(if this limit exists)}.
\]
In \hyperref[puttingitalltogether]{Section~\ref*{puttingitalltogether}}, we prove the following theorem.
\begin{final}\label{final}
Let $k\geq 2$ be an integer. For any $\epsilon>0$ and any $M\in\Z_{>0}$, there exist infinitely many sets of integers $\mathcal{M}$ of size $M$ such that
\[
\mu\lp\lb p\in\mathcal{P}
\mid\lb x^k+m\mid m\in\mathcal{M}\rb\text{ is dynamically distinguishable}\imod{p}\rb\rp
>1-\epsilon.
\]
\end{final}
\noindent Establishing the truth of the suggestion of~\cite{Ketal} mentioned above would immediately produce the $k=2$ case of \hyperref[final]{Theorem~\ref*{final}} as a weaker corollary.

For any $f,g\in\Z[x]$ and $p\in\mathcal{P}$, the dynamical systems $\lp[f]_p,\F_p\rp$ and $\lp[g]_p,\F_p\rp$ are isomorphic in the category of dynamical systems on the set $\F_p$ if and only if $f$ and $g$ are dynamically indistinguishable mod $p$.
In more generality, for any set $S$ and set maps $f,g:S\to S$, note that $\Gamma(S,f)\simeq\Gamma(S,g)$ if and only if there exists a bijective set map $\varphi:S\to S$ such that $\varphi\circ f=g\circ \varphi$.
In many settings, researchers study subcategories of the category of dynamical systems on the set $S$ by insisting that the maps $f,g$, and $\varphi$ belong to the set of morphisms in an appropriate category containing $S$ as an object.
For example, suppose $K$ is a field, $S=\P^1(K)$, and $f,g:S\to S$ are rational functions.
Then in the subcategory of dynamical systems of $\P^1(K)$, with the self-maps of $\P^1(K)$ restricted to rational maps, the dynamical systems $(\P^1(K),f)$ and $(\P^1(K),g)$ are isomorphic if and only if there exists a M\"obius transformation $\varphi$ such that $\varphi\circ f=g\circ \varphi$.
Fixing an integer $d\in\Z_{>1}$, setting $\mathcal{F}$ to be rational functions of degree $d$, and studing $M\lp\P^1(K),\mathcal{F}\rp$ leads to an interesting moduli space problem, one studied by Silverman in~\cite{SilvermanModuli} using geometric invariant theory.
See~\cite{BCE},~\cite{DeMarco}, and ~\cite{Levy} for further work on this problem and extensions of it.


To prove \hyperref[final]{Theorem~\ref*{final}}, we will distinguish dynamical systems by their periodic points.
If $(S,f)$ is a dynamical system, let $f^n=(\overbrace{f\circ\cdots\circ f}^{n\text{ times}})$ for any $n\in\Z_{>0}$.
If $s\in S$ has the property that there is some $n\in\Z_{>0}$ with $f^n(s)=s$, we say that $s$ is \emph{periodic} or a \emph{periodic point} of $(S,f)$.
The smallest such $n$ is the \emph{period} of $s$.
As is standard, we will also refer to points of period one as \emph{fixed points}.
Points of period $n$ are precisely those that lie in cycles of length $n$ in the graph $\Gamma(S,f)$. 
Periodic points are a classical object of study in discrete dynamical systems over $\C$, going back at least to work of Fatou~\cite{Fatou1,Fatou2} and Julia~\cite{Julia} in the early 20th century. Recently there has been much work on statistics of periodic points in families of dynamical systems over finite fields, partially motivated by an attempt started by Bach~\cite{Bach} to make rigorous the heuristic assumptions in Pollard's ``rho method" for integer factorization~\cite{Pollard}. For example, in~\cite{FG}, Flynn and the second author prove that for the family of polynomials in $\F_q[x]$ of a fixed degree $d$, the average number of cycles in their associated functional graphs is at least $\frac{1}{2}\log{q}-4$, as long as $d\geq\sqrt{q}$.
More recently, Bellah, the second author, et.\ al.~\cite{BGTW} develop a heuristic that implies that this average is $\frac{1}{2}\log{q}+O(1)$ for any $d$.
In~\cite{BS}, Burnette and Schmutz prove, for this same family of polynomials, that if $d=o\lp\sqrt{q}\rp$ as $d,q\to\infty$, then the average ``ultimate period'' of the associated functional graphs is at least $\frac{d}{2}\lp1+o(1)\rp$.

Our proof of \hyperref[final]{Theorem~\ref*{final}} relies on the trivial observation that for any $n\in\Z_{>0}$, if one directed graph has a cycle of length $n$ and another does not, then the graphs are not isomorphic.
As an illustration of our approach, consider the following example.
\begin{fixedpoints}\label{fixedpoints}
Let $f=x^2+1$ and $g=x^2+2$.
If $p\in\mathcal{P}$, then $\Gamma_{f,p}$ has a point of period one if and only if there exists $\alpha\in\F_p$ such that
\[
0=[f]_p(\alpha)-\alpha=\alpha^2+1-\alpha.
\]
Now, such an $\alpha$ exists if and only if the prime ideal $\lp p\rp\subseteq\Z$ splits (or ramifies) in the splitting field of $f(x)-x=x^2-x+1$ (over $\Q$).
Similarly, $\Gamma_{g,p}$ has a fixed point if and only if $\lp p\rp$ splits (or ramifies) in the splitting field of $g(x)-x$.
Let $K_f$ and $K_g$ be the splitting fields of $f(x)-x$ and $g(x)-x$, respectively.
The \ref{FDT} implies that the natural density of primes that split in $K_f$ and $K_g$ is the proportion of their Galois groups that fix a root of the polynomials whose roots we adjoin (that is, a root of $f(x)-x$ and $g(x)-x$, respectively).
Since $\Gal{\lp K_f/\Q\rp}\simeq\Gal{\lp K_g/\Q\rp}\simeq\Z/2\Z$, the natural density of primes that split in these fields is $\frac{1}{2}$.
Moreover, since $K_f$ and $K_g$ are linearly disjoint, we know that $\Gal{\lp K_fK_g/\Q\rp}\simeq\Z/2\Z\times\Z/2\Z$; thus, when we apply the theorem to the polynomial $\lp f(x)-x\rp\lp g(x)-x\rp$, we see that the splitting behavior of prime ideals in these two fields is independent.
%
%
That is,
\begin{align*}
\mu\lp p\in\mathcal{P}\mid\lb f,g\rb\text{is dynamically distinguishable}\imod{p}\rp\hspace{-200px}\\
&=\mu\lp p\in\mathcal{P}\mid\Gamma_{f,p}\nsimeq\Gamma_{g,p}\rp\\
&\geq\mu\lp p\in\mathcal{P}\mid
\Gamma_{f,p}\text{ has a fixed point and }\Gamma_{g,p}\text{ does not}\rp\\
&\hspace{20px}+\mu\lp p\in\mathcal{P}\mid
\Gamma_{f,p}\text{ does not have a fixed point and }\Gamma_{g,p}\text{ does}\rp\\
&=\frac{1}{2}\lp1-\frac{1}{2}\rp+\lp1-\frac{1}{2}\rp\frac{1}{2}\\
&=\frac{1}{2}.
\end{align*}
\end{fixedpoints}

The goal of this paper is to generalize this argument to points of period greater than one.
However, to produce polynomials in $\Z[x]$ and apply the \ref{FDT}, as in \hyperref[fixedpoints]{Example~\ref*{fixedpoints}}, we must prove several theorems to overcome various obstacles.
Before describing them, we introduce the notational conventions we will use throughout the rest of the paper.
If $F$ is a field and $f\in F[x]$, we will write $\Gal(f/F)$ to denote the Galois group of the splitting field of $f$ over $F$. 
Additionally, if $\mathcal{F}$ is a finite subset of $F[x]$, say with splitting fields $\lb K_f\rb_{f\in\mathcal{F}}$, then we will write $\prod_{f\in\mathcal{F}}{K_f}$ for the splitting field of $\prod_{f\in\mathcal{F}}{f}$.
(Of course, if we choose an algebraic closure of $F$, then $\prod_{f\in\mathcal{F}}{K_f}$ is isomorphic to the compositum of the images of the embeddings of the $K_f$s in that algebraic closure.)
Similarly, for any family of groups $\mathcal{G}$, we will write $\prod_{G\in\mathcal{G}}{G}$ for their direct product (if $\mathcal{G}=\lb G_1,\ldots G_n\rb$ for a positive integer $n$, we will write $G_1\times\cdots\times G_n$ for this group, and if there is some group $G$ such that $G_i=G$ for all $i\in\lb1,\ldots,n\rb$, we will write $G^n$.)
The following fact, which we will use often, relates these conventions: if $F$ is a field and $\mathcal{F}$ is finite subset of $F[x]$, say with splitting fields $\lb K_f\rb_{f\in\mathcal{F}}$, then the members of $\lb K_f\rb_{f\in\mathcal{F}}$ are pairwise $F$-linearly disjoint if and only if
\[
\Gal{\lp\lp\prod_{f\in\mathcal{F}}{K_f}\rp/F\rp}\simeq\prod_{f\in\mathcal{F}}{\Gal{\lp f/F\rp}}.
\]
Now, if $G$ is a group and $S_r$ is the symmetric group on $r$ letters, we write $G\wr S_r$ to mean the wreath product $G\wr_{\{1,\dots,r\}} S_r$.
That is, $G\wr S_r=G^r\rtimes S_r$, where $S_r$ acts on $G^r$ by permuting coordinates.
In particular, we note that $\lv G\wr S_r\rv=r!|G|^r$. See~\cite[Chapter 3A]{IsaacsGroupTheory} for background on the wreath product.
(In \hyperref[wreaths]{Section~\ref*{wreaths}}, we introduce and analyze the aspects of the wreath product that we require for this paper.)

With these notations in hand, we can now describe the path to generalizing \hyperref[fixedpoints]{Example~\ref*{fixedpoints}}.


\begin{itemize}
\item
If $K$ is a field and $f\in K[x]$, then $\alpha\in K$ is a fixed point in $\lp K,f\rp$ if and only if $\alpha$ is a root of $f(x)-x$.
To generalize the argument of \hyperref[fixedpoints]{Example~\ref*{fixedpoints}}, we review the famous ``dynatomic polynomials of $f$'' in \hyperref[galois]{Section~\ref*{galois}}, which we will denote by $\Phi_{f,n}$ for any $n\in\Z_{>0}$.
These polynomials have the property that for any $n\in\Z_{>0}$, every point of period $n$ in $\lp K,f\rp$ is a root of $\Phi_{f,n}$ (in particular, $\Phi_{f,1}=f(x)-x$).
When $K$ is the rational function field $\Q(c)$, Morton~\cite[Theorem~D]{MorGalGroups} proved that if $f(x)=x^k+c$ for some $k\in\Z_{>1}$, then for any $n,n^\prime\in\Z_{>0}$ with $n\neq n^\prime$, the splitting fields of $\Phi_{f,n}$ and $\Phi_{f,n^\prime}$ are linearly disjoint.
In \hyperref[genmorton]{Theorem~\ref*{genmorton}}, we generalize Morton's theorem to prove that for any $k,M,N\in\Z_{>1}$, there exist infinitely many sets of integers $\mathcal{M}$ of size $M$ such that for any $f,g\in\lb x^k+(c+m)\mid m\in\mathcal{M}\rb\subseteq\Q(c)[x]$ and $n,n^\prime$ with $n,n^\prime\leq N$, the splitting fields of $\Phi_{f,n}$ and $\Phi_{g,n^\prime}$ are linearly disjoint.
We point out that this includes the case where $n=n^\prime$, which is quite important for our applications.

\item
In \hyperref[fixedpoints]{Example~\ref*{fixedpoints}}, we set $f(x)=x^2+1$, and applied the \ref{FDT} to $\Gal{\lp\Phi_{f,1}/\Q\rp}\simeq\Z/2\Z$.
In general, the Galois groups of dynatomic polynomials are quite often wreath products of the form $\Z/n\Z\wr S_r$ for $n,r\in\Z_{>0}$.
To apply the \ref{FDT}, we must study the action of these wreath products on the roots of dynatomic polynomials.
In \hyperref[fixedpointsinwreaths]{Theorem~\ref*{fixedpointsinwreaths}}, we prove that for any $n,r\in\Z_{>0}$, the proportion of the group $\Z/n\Z\wr S_r$ (considered with its natural action on $\Z/n\Z\times\lb1,\ldots,r\rb$) that acts with a fixed point is approximately $1-e^{-\frac{1}{n}}$.
\item
In \hyperref[fixedpoints]{Example~\ref*{fixedpoints}}, with $f(x)=x^2+1$, we used the fact that for any $p\in\mathcal{P}$, the polynomial $\lc f(x)-x\rc_p$ has a root if and only if $\lp\F_p,[f]_p\rp$ has a fixed point.
Unfortunately, the picture is not quite so clear for points of period greater than one.
For example, if we let $g(x)=x^2+3$, then $\lc\Phi_{g,2}\rc_5$ has exactly one root (with multiplicity two), which happens to have period one in $\lp\F_5,[g]_5\rp$.
In \hyperref[rootsandcycles]{Corollary~\ref*{rootsandcycles}}, we provide a sufficient condition on $f\in\Z[x]$ and $n\in\Z_{>0}$ that ensures that $\lc\Phi_{f,n}\rc_p$ has a root in $\F_p$ if and only if $\lp\F_p,[f]_p\rp$ has a point of period $n$ for all but finitely many primes $p$.
\item
Finally, in \hyperref[puttingitalltogether]{Section~\ref*{puttingitalltogether}}, we apply the \ref{HIT} to the polynomials produced in \hyperref[genmorton]{Theorem~\ref*{genmorton}} to prove \hyperref[final]{Theorem~\ref*{final}}.
\end{itemize}

\section{Galois groups of dynatomic polynomials}
\label{galois}

As we intend to distinguish dynamical systems by analyzing their periodic points, we will make use of the theory of dynatomic polynomials (and their Galois groups).
See \cite{MP}, \cite{MorCurves} (and the correction in \cite{MorCorrection}), \cite{MorGalGroups}, and ~\cite[Chapter 4.1]{SilvermanADS} for background in this area.
We sketch an introduction, focusing on the aspects of the theory we will use in our results.

Let $K$ be a field, $f\in K[x]$, and $n\in\Z_{>0}$.
The points of period $n$ of the dynamical system $\lp K,f\rp$ are certainly roots of the polynomial $f^n(x)-x$.
However, if $d\in\Z_{>0}$ and $d\mid n$, then this polynomial vanishes on points of period $d$ as well (for example, if $\alpha\in K$ is a fixed point of $(K,f)$, i.e. $f(\alpha)=\alpha$, then $f^n(\alpha)=\alpha$ for all $n\in\Z_{>0}$).
In an attempt to sieve out the points of lower period, one defines the $n$th \emph{dynatomic polynomial of $f$} for any $n\in\Z_{>0}$:
\[
\Phi_{f,n}(x)
:=\prod_{d\mid n}
{\lp f^d(x)-x\rp^{\mu(n/d)}},
\]
where $\mu:\Z_{\geq0}\to\lb-1,0,1\rb$ is the usual M\"obius function.
The fact that 
\[
\prod_{d\mid n} \Phi_{f,n}(x) = f^n(x)-x
\]
follows quickly by applying the M\"obius inversion formula.
As usual, we omit ``$K$'' from the notation ``$\Phi_{f,n}$''; we will always specify the set of coefficients of $f$, so that the field $K$ will be clear from context.
As indicated by its name, the $n$th dynatomic polynomial is analogous to the $n$th cyclotomic polynomial, which vanishes precisely on primitive $n$th roots of unity.
(As mentioned in the discussion following \hyperref[fixedpoints]{Example~\ref*{fixedpoints}}, it turns out that $\Phi_{f,n}$ may occasionally vanish on points of period $d$ for $d<n$: see~\cite[Example 4.2]{SilvermanADS}.
In \hyperref[rootsandcycles]{Corollary~\ref*{rootsandcycles}}, we address this inconvenience.)
We should mention that it is not \emph{a priori} obvious that $\Phi_{f,n}$ is a polynomial.
See~\cite[Theorem~2.5]{MP} for a proof that $\Phi_{f,n}\in K[x]$.
(In particular, if $f\in\Z[x]$ and $f$ is monic, then $\Phi_{f,n}\in\Z[x]$ by Gauss's Lemma.)
The degrees of certain dynatomic polynomials will be important quantities in many computations that follow, so we introduce the following notation.
\begin{rkn}\label{rkn}
For any $n\in\Z_{>0}$ and $k\in\Z_{>1}$, let
\[
r_k(n)=\frac{1}{n}\cdot
\sum_{d\mid n}
{k^d\mu\lp\frac{n}{d}\rp}.
\]
\end{rkn}
\noindent Note that $nr_k(n)$ is the degree (in $x$) of the $n$th dynatomic polynomial of $x^k+c\in\Q(c)[x]$.

As mentioned in \hyperref[fixedpoints]{Example~\ref*{fixedpoints}}, our proof of \hyperref[final]{Theorem~\ref*{final}} relies in part on the knowledge of the structure of the Galois groups of $\Phi_{f,n}$, where $n\in\Z_{>0}$ and $f(x)=x^k+m\in\Z[x]$ for $k\in\Z_{>1}$ and $m\in\Z$.
Moreover, we must find arbitrarily large finite sets of polynomials of this form that have the property that the splitting fields of their dynatomic polynomials are linearly disjoint.
For a specific polynomial $f\in\Z[x]$ of this form and any large $n$, it is difficult to compute the Galois group of $\Phi_{f,n}$, since the degree of $\Phi_{f,n}$ is so large, but---thanks to work of Morton~\cite[Theorem~D]{MorGalGroups}---the Galois groups of $\Phi_{f,n}$ for $f(x)=x^k+c\in\Q(c)[x]$ are known.
The remainder of this section addresses the question of linear disjointness in the function field setting.

We will need the following elementary lemma of field theory.

\begin{automorphismlemma}\label{automorphismlemma}
Let $K$ be a field and let $\sigma\in\Aut(K)$.
Let $f\in K[x]$ be an irreducible polynomial, and let $f^\sigma$ be the polynomial in $K[x]$ obtained by applying $\sigma$ to each of the coefficients of $f$.
Let $L,L^\sigma$ be the splitting fields of $f,f^\sigma$, respectively.
Then $L$ and $L^\sigma$ are isomorphic as fields.
In particular,
\begin{enumerate}
\item $\Gal(f/K)\simeq \Gal(f^\sigma/K)$, and
\item if $K$ is the fraction field of a Dedekind domain and $\mathfrak{p}$ is a prime of $K$, then
\[
\text{$\mathfrak{p}$ ramifies in $L$ if and only if $\sigma(\mathfrak{p})$ ramifies in $L^\sigma$.}
\]
\end{enumerate}
\end{automorphismlemma}
\begin{proof}
Let $\overline{K}$ be an algebraic closure of $K$ containing both $L$ and $L^\sigma$.
Then we can extend $\sigma\in\Aut(K)$ to some automorphism $\widehat{\sigma}\in\Aut(\overline{K})$~\cite[Theorem V.2.2.8]{Lang}.
It is easy to see that $\widehat{\sigma}$ furnishes a one-to-one correspondence between the roots of $f$ and the roots of $f^{\sigma}$; thus $\widehat{\sigma}|_L:L\to L^\sigma$ is an isomorphism. Statement (1) follows immediately, and the map from $\Gal(L/K)$ to $\Gal(L^\sigma/K)$ is given by
\[
\tau\mapsto
\widehat{\sigma}^{\hspace{1px}-1}\circ\tau\circ\widehat{\sigma}.
\]
For (2), if the prime $\mathfrak{p}$ of $K$ ramifies in $L$, there is a prime $\mathfrak{q}$ of $L$ with $e(\mathfrak{q}/\mathfrak{p})>1$, and 
\[
e(\widehat{\sigma}(\mathfrak{q})/\sigma(\mathfrak{p}))=e(\widehat{\sigma}(\mathfrak{q})/\widehat{\sigma}(\mathfrak{p}))=e(\mathfrak{q}/\mathfrak{p})>1,
\]
so $\sigma(\mathfrak{p})$ ramifies in $L^\sigma$. Replacing $\widehat{\sigma}$ by its inverse shows that the converse holds as well.
\end{proof}

For the rest of this section, we will work with polynomials $f(x)\in\Q(c)[x]$.
For any $n\in\Z_{>0}$, let
\begin{itemize}
\item $\Sigma_{f,n}$ denote the splitting field of $\Phi_{f,n}$, and
\item $K_{f,n}$ denote the splitting field of $f^n(x)-x$.
\end{itemize}
These splitting fields will be defined over $\Q(c)$ or $\overline{\Q}(c)$, depending on context. There should be no ambiguity about which definition is intended. Note that in either case, $K_{f,n}$ is the compositum of the fields $\Sigma_{f,d}$ for all positive integers $d$ dividing $n$. 

The next theorem generalizes the first part of Theorem~D in~\cite{MorGalGroups}.


\begin{genmorton}\label{genmorton}
Let $k\geq 2$ be an integer and $f=f(x) = x^k+c\in\Q(c)[x]$.
Suppose that $M,N\in\Z_{>0}$.
Then there exist infinitely many $M$-tuples of integers $(m_1,\ldots m_M)\in\Z^M$ such that
\[
\Gal{\lp\lp\prod_{i=1}^M{K_{f+m_i, N}}\rp\bigg/\overline{\Q}(c)\rp}
\simeq\prod_{i=1}^M{\Gal{\lp K_{f+m_i, N}\hspace{1px}\big/\hspace{1px}\overline{\Q}(c)\rp}}
\]
\end{genmorton}
\begin{proof}

Following the proof of Theorem~10 in~\cite{MorGalGroups}, for any $n\in\Z_{>0}$, there exists a polynomial $\delta_n(x)\in\Z[x]$ such that the finite primes in $\overline{\Q}(c)$ that ramify in $\Sigma_{f,n}$ have the form $c-b$, where $b\in\overline{\Q}$ satisfies $\delta_n(b)=0$. The roots of $\delta_n(x)$ are the roots of the hyperbolic components of the degree-$k$ Multibrot set, which is the famous Mandelbrot set when $k=2$.
It is a consequence of the structure of the Multibrot set that $\delta_n(x)$ and $\delta_d(x)$ have no roots in common if $d\neq n$ (closures of hyperbolic components of different periods may only intersect at a root of the component of higher period, see~\cite{Branner} and~\cite{Schleicher}.)
For any $m\in\Z$, consider the unique $\sigma\in\Aut(\overline{\Q}(c)/\overline{\Q})$ defined by $\sigma(c)=c+m$. Then $f+m=f^\sigma$ in the notation of \hyperref[automorphismlemma]{Lemma~\ref*{automorphismlemma}}, so the primes that ramify in $\Sigma_{f+m,n}$ have the form $c-\lp b-m\rp$, where $b\in\overline{\Q}$ satisfies $\delta_n(b)=0$. 

With the above facts in mind, let $R$ be the (finite) set
\[
\lb b\in\overline{\Q}\mid\text{there exists }d\in\Z_{>0}\text{ such that }d\mid N\text{ and }\delta_d(b)=0\rb,
\]
then choose $\lp m_1,\ldots,m_M\rp\in\Z^M$ such that the sets $\lb R-m_i\rb$ are pairwise disjoint.
As $R$ is a finite set, there are infinitely many such choices.
For any $i\in\lb1,\ldots,M\rb$, let
\[
\mathcal{F}=\lb\Sigma_{f+m_i,d}
\,\,\Big\vert\,\,d\in\Z_{>0}\text{ with }d\mid N\rb
\hspace{10px}\text{and}\hspace{10px}
\mathcal{G}=\lb\Sigma_{f+m_{j},d}
\,\,\Big\vert\,\,\substack{
j\in\lb1,\ldots,M\rb
\text{ with }j\neq i\\
d\in\Z_{>0}\text{ with }d\mid N}\rb.
\]
Recall that for any $m\in\Q$ and $n\in\Z_{>0}$, we have $K_{f+m,n}=\prod_{d\mid n} \Sigma_{f+m,d}$. Thus 
\[
\prod_{F\in\mathcal{F}}F= K_{f+m_i,N}\hspace{10px}\text{ and }\hspace{10px}
\prod_{F\in\mathcal{G}}F = \prod_{\substack{j=1 \\ j\neq i}}^M K_{f+m_j,N}.
\]
By our choice of the $m_i$s, these two fields have no finite ramified primes in common, so they are linearly disjoint over $\overline{\Q}(c)$. Therefore the fields $K_{f+m_1,N},\dots,K_{f+m_M,N}$ are linearly disjoint over $\overline{\Q}(c)$.
The result now follows by elementary Galois theory.\end{proof}

The corollary below follows immediately from \hyperref[genmorton]{Theorem~\ref*{genmorton}} and by work of Morton.
It will be crucial in the proof of \hyperref[final]{Theorem~\ref*{final}}.

\begin{QbarQ}\label{QbarQ}
Keep the same hypotheses as \hyperref[genmorton]{Theorem~\ref*{genmorton}}, and for any $\mathbf{m}=\lp m_1,\ldots,m_M\rp\in\Z^M$, let
\[
\mathcal{F}\lp\mathbf{m}\rp
=\lb\Sigma_{f+m_i,d}\mid i\in\lb1,\ldots,M\rb\text{and }d\in\Z_{>0}\text{ such that }d\mid N\rb.
\]
Then there exist infinitely many $\mathbf{m}\in\Z^M$ such that
\begin{itemize}
\item
any field in $\mathcal{F}\lp\mathbf{m}\rp$ is linearly disjoint from the compositum of the others, 
\item
if $\Sigma_{f+m_i,d}\in\mathcal{F}(\mathbf{m})$, then $\Gal{\lp \Sigma_{f+m_i,d}/\Q(c)\rp}\simeq\Gal{\lp \Sigma_{f+m_i,d}/\overline{\Q}(c)\rp}\simeq\lp\Z/d\Z\wr S_{r_k(d)}\rp$, and
\item
$\Gal{\lp\lp\prod_{i=1}^M{K_{f+m_i, N}}\rp\Big/\Q(c)\rp}\simeq \prod_{i=1}^N{\prod_{d\mid N}{\lp\Z/d\Z\wr S_{r_k(d)}\rp}}$.
\end{itemize}
(Recall that $dr_k(d)$ is the degree of the $d$th dynatomic polynomial of $f(x)$, see \hyperref[rkn]{Definition~\ref*{rkn}}.)
\end{QbarQ}
\begin{proof}
Theorem 9 in~\cite{MorGalGroups} shows that $f(x)=x^k+c\in\Q(c)[x]$ satisfies the assumptions of Theorem B in the same paper, which proves that for any $n\in\Z_{>0}$, both $\Gal(\Phi_{f,n}/\Q(c))$ and $\Gal(\Phi_{f,n}/\overline{\Q}(c))$ are isomorphic to $\Z/d\Z\wr S_{r_k(d)}$.
Applying \hyperref[automorphismlemma]{Lemma~\ref{automorphismlemma}}, with $\sigma:c\mapsto c+m$, we see  that the same is true of the Galois group of $\Phi_{f+m,n}$ for any $m\in\Q$.

Let $\mathbf{m}=\lp m_1,\ldots,m_M\rp$ be any of the (infinitely many) $M$-tuples that satisfy the conclusion of \hyperref[genmorton]{Theorem~\ref*{genmorton}}.
From the proof of \hyperref[genmorton]{Theorem~\ref*{genmorton}}, we know that if $i,j$ are distinct integers in $\lb1,\ldots,M\rb$ and $d$ is a positive integer divisor of $N$, then $\Sigma_{f+m_i,d}$ and $\Sigma_{f+m_j,d}$ are linearly disjoint over $\overline{\Q}(c)$. Thus
\begin{align*}
\Gal{\lp\lp\prod_{i=1}^M
{K_{f+m_i, N}}\rp\bigg/\overline{\Q}(c)\rp}
&\simeq\prod_{i=1}^M
{\Gal{\lp K_{f+m_i, N}
\hspace{1px}\big/\hspace{1px}
\overline{\Q}(c)\rp}}\\
&\simeq\prod_{i=1}^M
{\prod_{d\mid N}
{\Gal{\lp\Sigma_{f+m_i, d}
\hspace{1px}\big/\hspace{1px}
\overline{\Q}(c)\rp}}}\\
&\simeq\prod_{i=1}^M
{\prod_{d\mid N}
{\lp\Z/d\Z\wr S_{r_k(d)}\rp}}
\end{align*}
Let $G=\Gal{\lp\lp\prod_{i=1}^M{K_{f+m_i, N}}\rp\bigg/\Q(c)\rp}$.
By Theorem~B from~\cite{MorGalGroups} again, we know $G$ is isomorphic to a subgroup of $\prod_{i=1}^M{\prod_{d\mid N}{\lp\Z/d\Z\wr S_{r_k(d)}\rp}}$.
Conversely, since $\overline{\Q}(c)$ contains $\Q(c)$, we see that $\prod_{i=1}^M{\prod_{d\mid N}{\lp\Z/d\Z\wr S_{r_k(d)}\rp}}$ is isomorphic to a subgroup of $G$, so the proof is complete.
\end{proof}

\section{Fixed point proportions in wreath products}
\label{wreaths}

In this section, we analyze some statistics of a certain family of wreath products.
As these groups appear as Galois groups of dynatomic polynomials, these statistics are a vital component of our proof of \hyperref[final]{Theorem~\ref*{final}}.
We begin with some definitions.

Suppose that $n,r\in\Z_{>0}$.
Recall the definition of $\Z/n\Z\wr S_r$ from the end of \hyperref[intro]{Section~\ref*{intro}}. Let $B(n,r)$ denote $\Z/n\Z\times\lb1,\ldots,r\rb$. The group $\Z/n\Z\wr S_r$ acts on the set $B(n,r)$; concretely, for any $\sigma=\lp\lp\overline{a_1},\dots,\overline{a_r}\rp,\pi\rp\in\Z/n\Z\wr S_r$, this action is
\begin{align*}
\sigma:B(n,r)&\to B(n,r)\\
\lp\overline{b},i\rp&\mapsto\lp\overline{b+a_i},\pi(i)\rp.
\end{align*}
For any $\sigma\in\Z/n\Z\wr S_r$, define

\[\Fix{\sigma}=\lb\lp\overline{b},i\rp\in B(n,r)
\,\,\big\vert\,\,\sigma\lp\overline{b},i\rp=\lp\overline{b},i\rp\rb;
\]
then we set
\[
P_{r,n}=\frac{\lv\lb\sigma\in\Z/n\Z\wr S_r\mid\Fix{\sigma}\neq\emptyset\rb\rv}{\lv\Z/n\Z\wr S_r\rv}.\]

In many cases, this action matches the action of the Galois groups of dynatomic polynomials on the roots of those polynomials, so we make the following definition.
\begin{Pkn}\label{Pkn}
For any $k\in\Z_{>1}$ and $n\in\Z_{>0}$, let
\[
P_k(n)=P_{r_k(n),n},
\]
where $r_k(n)=\sum_{d\mid n}
{k^d\mu\lp\frac{n}{d}\rp}$ as in \hyperref[rkn]{Definition \ref*{rkn}}.
\end{Pkn}

\begin{GaloisWreathActions}\label{GaloisWreathActions}
When we apply the results of this section in the proof of \hyperref[final]{Theorem~\ref*{final}}, the groups $\Z/n\Z\wr S_{r_k(n)}$ will be isomorphic to the groups $\Gal(\Phi_{f,n}/\Q)$ in a setting where $f\in\Z[x]$ and the roots of $\Phi_{f,n}$ are exactly the $nr_k(n)$ points of period $n$ in $\lp\overline{\Q},f\rp$.
In this setting, we can identify $B\lp n,r_k(n)\rp$ with the union of the $r_k(n)$ cycles of length $n$ in $(\overline{\Q},f)$ in such a way that the permutation action of $\Gal{\lp\Phi_{f,n}/\Q\rp}$ on the roots of $\Phi_{f,n}$ is precisely the action of $\Z/n\Z\wr S_{r_k(n)}$ on $B(n,r)$ described above (see Section 4 of \cite{MP} for details). 
In particular, in the proof of \hyperref[final]{Theorem~\ref*{final}}, we will exploit the fact that
\[
P_k(n) = \frac{\lv\sigma\in\Gal(\Phi_{f,n}/\Q)\mid\text{$\sigma$ fixes a root of $\Phi_{f,n}$}\rv}{\lv\Gal(\Phi_{f,n}/\Q)\rv}
\]
for the polynomials $f\in\Z[x]$ and integers $n\in\Z_{>0}$ under consideration.
\end{GaloisWreathActions}

Now, the Galois groups in the conclusion of \hyperref[QbarQ]{Corollary~\ref*{QbarQ}} are isomorphic to direct sums of the wreath products defined above.
With this in mind, we need a bit more notation before proceeding---notation whose purpose will become clear in the proof of \hyperref[final]{Theorem~\ref*{final}}.

If $G,H$ are groups acting on sets $B,C$, say with actions $\odot_G,\odot_H$, respectively, define the \emph{product action} of $G\times H$ on $B\times C$ to be the action
\begin{align*}
\lp G\times H\rp\times\lp B\times C\rp
&\to B\times C\\
\lp(g,h),(b,c)\rp&\mapsto(g,h)\odot_{G\times H}(b,c):=\lp g\odot_Gb,h\odot_Hc\rp.
\end{align*}
Suppose $k\in\Z_{>1}$ and let $A=\lp b_i\rp_{i\in\Z_{>0}}$ be any increasing arithmetic progression of positive integers.
For any $i\in\Z_{>0}$, define
\[
W_{A,i}=\Z/b_i\Z\wr S_{r_k\lp b_i\rp}
\times\Z/b_i\Z\wr S_{r_k\lp b_i\rp}
\hspace{10px}\text{and}\hspace{10px}
B_{A,i}= B\lp b_i,r_k\lp b_i\rp\rp
\times B\lp b_i,r_k\lp b_i\rp\rp,
\]
so that $W_{A,i}$ acts on $B_{A,i}$ with the product action defined above.
Next, for any $n\in\Z_{>0}$, let
\[
W_A\lp n\rp=W_{A,1}\times\cdots\times W_{A,n}
\hspace{10px}\text{and}\hspace{10px}
B_A\lp n\rp=B_1\times\cdots\times B_{n};
\]
once again, $W_A\lp n\rp$ acts on $B_A\lp n\rp$ with the product action induced from the action of the $W_{A,i}$s on the $B_{A,i}$s.
In the proof of \hyperref[final]{Theorem~\ref*{final}}, we require knowledge of the proportion of these groups that act with a fixed point.
To begin specifying the quantity we need, we first set, for any $i\in\lb1,\ldots, n\rb$, 
\[
C_{A,i}\lp n\rp=\lb\lp (\sigma_1,\tau_1),\dots,(\sigma_n,\tau_n)\rp\in W_A\lp n\rp\mid
\text{exactly one of $\Fix{\sigma_i},\Fix{\tau_i}$ is empty}\rb.
\]
Let $s_{A,0}=0$. Define 
\[
s_{A,n}=\frac{\lv\bigcup_{i=1}^n{C_{A,i}\lp n\rp}\rv}{\lv W_A\lp n\rp \rv}.
\]
The main technical result of this section is \hyperref[recurrence]{Corollary~\ref*{recurrence}}, which exhibits a recurrence relation on the terms of sequences of the form $s_{A,n}$ and computes the limit of this sequence; the recurrence relation uses the quantities $P_k(b)$, for $b\in A$---these quantities were defined in in \hyperref[Pkn]{Definition~\ref*{Pkn}}.
We defer the proof until the end of the section, after establishing some estimates on fixed-point proportions in wreath products.
\begin{recurrence}\label{recurrence}
If $k\in\Z_{>1}$ and $A=\lp b_i\rp_{i\in\Z_{>0}}$ is any increasing arithmetic progression of positive integers, then for any $n\in\Z_{>0}$,
\[
s_{A,n}=s_{A,n-1}+\lp1-s_{A,n-1} \rp 2P_k\lp b_n\rp\lp1-P_k\lp b_n\rp\rp.
\]
Moreover, $\lim_{n\to\infty} s_{A,n}=1$.
\end{recurrence}

We turn to computing $P_{r,k}$ for general $r\in\Z_{>0}$ and $k\in\Z_{>1}$.
To do so, we recall the \emph{rencontres numbers} from combinatorics.
For any $r\in\Z_{>0}$ and $i\in\lb0,\ldots,r\rb$, we will denote the $(r,i)$th \emph{rencontres number} by $D_{r,i}$; that is, $D_{r,i}$ is the number of permutations of $\lb1,\ldots,r\rb$ with exactly $i$ fixed points.
In particular, the number of derangements of $\lb1,\ldots,r\rb$ is $D_{r,0}$.
For convenience, we set $D_{0,0}=1$.
We now record some basic identities involving rencontres numbers, which we will use in the proof of \hyperref[fixedpointsinwreaths]{Theorem~\ref*{fixedpointsinwreaths}}, below.
\begin{rencontresprop}\label{rencontresprop}
For all $i,r\in\Z_{\geq 0}$,
\begin{enumerate} 
\item $D_{r,i}=\binom{r}{i}D_{r-i,0}$ and 
\item $\sum_{i=1}^r{\binom{r}{i}D_{r-i,0}}=r!-D_{r,0}$.
\end{enumerate}
\end{rencontresprop}
\begin{proof}
For (1), note that a permutation of $\lb 1,\dots,r\rb$ with precisely $i$ fixed points is completely determined by choosing its $i$ fixed points and specifying its action on the $r-i$ remaining non-fixed points. For (2), observe that $\sum_{i=0}^r D_{r,i} = |S_r|=r!$, as each permutation in $S_r$ contributes to exactly one term in the sum, then apply (1).
\end{proof}

We now prove an important estimate on $P_{r,n}$ for all wreath products defined above (that is, a larger class of wreath products than those which arise as Galois groups of dynatomic polynomials).

\begin{fixedpointsinwreaths}\label{fixedpointsinwreaths}
Suppose that $n,r\in\Z_{>0}$.
Then
\[
\lv P_{r,n}-\lp1-e^{-\frac{1}{n}}\rp\rv
<\frac{1+2^r}{r!}.
\]
\end{fixedpointsinwreaths}
\begin{proof}
We begin by noting that if $\sigma\in\Z/n\Z\wr S_r$, then $\lv\Fix{\sigma}\rv$ is a multiple of $n$. This follows from the fact that if $\sigma$ fixes any $\lp\overline{b},i\rp\in B(n,r)$, then it must fix each $\lp\overline{c},i\rp$ for all $\overline{c}\in\Z/n\Z$.
Now, if $j\in\lb1,\ldots,r\rb$, $\sigma=\lp\lp\overline{a_i}\rp,\pi\rp$, and $\lv\Fix{\sigma}\rv=nj$, then $\pi$, acting on $\lb1,\ldots,r\rb$, has at least $j$ fixed points.
Moreover, there is a subset $R$ of the fixed points of $\pi$ such that
\begin{itemize}
\item $\lv R\rv=j$ and
\item if $i^\prime\in\lb1,\ldots,r\rb$ is a fixed point of $\pi$, then $i^\prime\in R$ if and only if $\overline{a_{i^\prime}}=0$.
\end{itemize}
In other words, if $\pi\in S_r$, $\Fix{\pi}=T$, and $\lp\overline{a_i}\rp\in\lp\Z/n\Z\rp^r$, then $\lv\Fix\lp\lp\overline{a_i}\rp,\pi\rp\rv=nj$ if and only if there exists $R\subseteq T$ with $\lv R\rv=j$ and for all $i^\prime\in T$, $a_{i^\prime}=\overline{0}$ if and only if $i^\prime\in R$.
Using this fact, and enumerating permutations $\pi$ by their number of fixed points, note that
\[
\lv\lb\sigma\in\Z/n\Z\wr S_r\mid\lv\Fix{\sigma}\rv=nj\rb\rv
=\sum_{i=1}^r{\binom{i}{j}D_{r,i}(n-1)^{i-j}n^{r-i}}.
\]
Using \hyperref[rencontresprop]{Lemma~\ref*{rencontresprop}}, we see that
\begin{align*}
\lv\lb\sigma\in\Z/n\Z\wr S_r\mid\Fix{\sigma}\neq\emptyset\rb\rv
&=\sum_{j=1}^r{\sum_{i=1}^r{\binom{i}{j}D_{r,i}}(n-1)^{i-j}n^{r-i}}\\
&=\sum_{j=1}^r{\sum_{i=1}^r{\binom{i}{j}\binom{r}{i}D_{r-i,0}}(n-1)^{i-j}n^{r-i}}\\
&=\sum_{i=1}^r{\binom{r}{i}D_{r-i,0}n^{r-i}\sum_{j=1}^r{\binom{i}{j}(n-1)^{i-j}}}\\
&=\sum_{i=1}^r{\binom{r}{i}D_{r-i,0}n^{r-i}\lp n^i-(n-1)^i\rp}\\
&=\sum_{i=1}^r{\binom{r}{i}D_{r-i,0}n^r}-\sum_{i=1}^r{\binom{r}{i}D_{r-i,0}n^{r-i}(n-1)^i}\\
&=r!n^r-D_{r,0}n^r-\sum_{i=1}^r{\binom{r}{i}D_{r-i,0}n^{r-i}(n-1)^i}\\
&=r!n^r-\sum_{i=0}^r{\binom{r}{i}D_{i,0}n^i(n-1)^{r-i}}.
\end{align*}
Thus,
\[
P_{r,n}
=\frac{1}{r!n^r}\lp r!n^r-\sum_{i=0}^r{\binom{r}{i}D_{i,0}n^i(n-1)^{r-i}}\rp
=1-\sum_{i=0}^r{\frac{D_{i,0}}{i!(r-i)!}\lp\frac{n-1}{n}\rp^{r-i}}.
\]
Using the Taylor expansion of $e^x$ evaluated at $x=1-\frac{1}{n}$, we see that
\[
\lv P_{r,n}-\lp1-e^{-\frac{1}{n}}\rp\rv
=\lv\frac{1}{e}\sum_{i=0}^\infty{\frac{1}{i!}\lp\frac{n-1}{n}\rp^i}
-\sum_{i=0}^r{\frac{D_{i,0}}{i!(r-i)!}\lp\frac{n-1}{n}\rp^{r-i}}\rv.
\]
Finally, we use the well-known fact that $\frac{i!}{e}-1<D_{i,0}<\frac{i!}{e}+1$ to conclude
\begin{align*}
\lv P_{r,n}-\lp1-e^{-\frac{1}{n}}\rp\rv
&<\frac{1}{e}\sum_{i=0}^\infty{\frac{1}{i!}\lp\frac{n-1}{n}\rp^i}
-\frac{1}{e}\sum_{i=0}^r{\frac{i!}{i!(r-i)!}\lp\frac{n-1}{n}\rp^{r-i}}
+\sum_{i=0}^r{\frac{1}{i!(r-i)!}\lp\frac{n-1}{n}\rp^{n-i}}\\
&\leq\frac{1}{(r+1)!}\lp\frac{n-1}{n}\rp^{r+1}
+\frac{1}{r!}\lp1+\frac{n-1}{n}\rp^r\\
&<\frac{1+2^r}{r!}.
\end{align*}
\end{proof}

We record a simple bound we will use in our study of fixed point proportions.
The goal is to prove that $P_k(n)(1-P_k(n))$ is close enough to $\frac{1}{n}$ to satisfy the hypotheses of \hyperref[comparativerecurrence]{Lemma~\ref*{comparativerecurrence}}, so the exact error bound does not matter much.


\begin{pnbound}\label{pnbound}
Suppose that $k\in\Z_{>0}$.
If $n\in\Z_{>0}$, then
\[
\lv P_k(n)\lp1-P_k(n)\rp-\frac{1}{n}\rv<\frac{121}{n^2}.
\]
\end{pnbound}
\begin{proof}
Using \hyperref[fixedpointsinwreaths]{Theorem~\ref*{fixedpointsinwreaths}}, we see that
\begin{align*}
\lv P_k(n)\lp1-P_k(n)\rp-e^{-\frac{1}{n}}\lp1-e^{-\frac{1}{n}}\rp\rv
&=\lv P_k(n)-\lp1-e^{-\frac{1}{n}}\rp+\lp1-e^{-\frac{1}{n}}\rp^2-P_n(k)^{2}\rv\\
&\leq\lv P_k(n)-\lp1-e^{-\frac{1}{n}}\rp\rv\cdot
\lp1+\lv\lp1-e^{-\frac{1}{n}}\rp+P_k(n)\rv\rp\\
&<\frac{1+2^{r_k(n)}}{r_k(n)!}\cdot3.
\end{align*}

Writing the Taylor series of $e^{x}\lp1-e^{x}\rp$ shows
\[
\lv e^{-\frac{1}{n}}\lp1-e^{-\frac{1}{n}}\rp-\frac{1}{n}\rv
<\frac{3}{2n^2},
\]
so by the triangle inequality,
\[
\lv P_k(n)\lp1-P_k(n)\rp-\frac{1}{n}\rv
<3\cdot\frac{1+2^{r_k(n)}}{r_k(n)!}+\frac{3}{2n^2}.
\]

Since $0<P_k(n)\lp1-P_k(n)\rp<\frac{1}{4}$, we know that $\lv P_k(n)\lp1-P_k(n)\rp-\frac{1}{n}\rv<1$ for all $n$; in particular, the statement is true for all $n\leq11$.
Thus, we may assume that $n\geq12$.
This implies immediately that $r_k(n)>\frac{k^{n-1}}{n}>7$ and $k^{n-1}>n^3$.
Next, we note that
\[
3\cdot\frac{1+2^x}{(x-1)!}<1
\hspace{10px}\text{for all}\hspace{10px}
x\geq7.
\]
Putting these estimates together, we obtain
\[
3\cdot\frac{1+2^{r_k(n)}}{(r_k(n))!}< \frac{1}{r_k(n)}<\frac{n}{k^{n-1}}<\frac{1}{n^2}
\hspace{10px}\text{when}\hspace{10px}
n\geq12.
\]
So for all such $n$, we conclude that
\[
\lv P_k(n)\lp1-P_k(n)\rp-\frac{1}{n}\rv
<\frac{1}{n^2}+\frac{3}{2n^2}=\frac{5}{2n^2}.
\]
\end{proof}

Before proving the corollary we will use in the proof of \hyperref[final]{Theorem~\ref*{final}}, we prove a short lemma about a certain class of recurrence relations.

\begin{comparativerecurrence}\label{comparativerecurrence}
Suppose $\lp a_n\rp_{n\in\Z_{>0}}$ is a sequence of real numbers that satisfies
\[
\sum_{n=1}^\infty{a_n}=\infty,
\hspace{10px}
\lim_{n\to\infty}{a_n}=0,
\hspace{10px}\text{and}\hspace{10px}
a_n\in[0,1]
\hspace{10px}\text{for all }\hspace{10px}
n\in\Z_{>0}.
\]
Suppose $t_0\in[0,1]$, and define
\[
t_n=t_{n-1}+a_n\lp1-t_{n-1}\rp
\hspace{10px}\text{for all}\hspace{10px}
n\in\Z_{>0}.
\]
Then $\lim_{n\to\infty}{t_n}=1$.
\end{comparativerecurrence}
\begin{proof}
A short induction argument shows that $\lp t_n\rp$ is nondecreasing and bounded above by 1.
So $\lp t_n\rp$ converges; suppose for a contradiction that it converges to $L\in[0,1)$.
Note that
\[
t_n-t_{n-1}=a_n\lp1-t_{n-1}\rp\geq a_n(1-L).
\]
Summing both sides of this inequality over all $n\in\Z_{\geq0}$ yields the contradiction
\[
\lim_{n\to\infty}{\lp t_n-t_0\rp}=\infty.
\]
\end{proof}

Putting together the results in this section, we can now prove \hyperref[recurrence]{Corollary~\ref*{recurrence}}.
\begin{proof}[Proof of \cref{recurrence}]

Recall that $W_A(n)=W_{A,1}\times\cdots\times W_{A,n}$, that
\[
C_{A,i}\lp n\rp=\lb\lp (\sigma_1,\tau_1),\dots,(\sigma_n,\tau_n)\rp\in W_A\lp n\rp\mid
\text{exactly one of $\Fix{\sigma_i},\Fix{\tau_i}$ is empty}\rb,
\]
and that $s_{A,n}$ is defined by $s_{A,0}=0$ and
\[
s_{A,n}=\frac{\lv\bigcup_{i=1}^n{C_{A,i}\lp n\rp}\rv}{\lv W_A\lp n\rp \rv}
\]
for $n>0$.
Observe that $\lv W_A(n)\rv = \lv W_A(n-1)\rv\lv W_{A,n}\rv$ and $\lv C_{A,i}(n)\rv=\lv C_{A,i}(n-1)\rv \lv W_{A,n}\rv$ for $1\leq i\leq n-1$.
Thus, the sequence $s_{A,n}$ satisfies the recurrence relation
\begin{align*}
s_{A,n} & = s_{A,n-1}\frac{\lv W_{A,n}\rv}{\lv W_{A,n}\rv} + \lp 1- s_{A,{n-1}}\rp \frac{\lv\lb (\sigma,\tau)\in W_{A,n}\mid \text{exactly one of Fix $\sigma$, Fix $\tau$ is empty}\rb\rv}{\lv W_{A,n}\rv}\\
& =s_{A,n-1}+\lp1-s_{A,n-1} \rp 2P_k\lp b_n\rp\lp1-P_k\lp b_n\rp\rp.
\end{align*}
Since $0<P_k\lp b_n\rp<1$ for all $n$, we note that $0<2P_k\lp b_n\rp\lp1-P_k\lp b_n\rp\rp<1$ for all $n$ as well.
Setting $a_n=2P_k\lp b_n\rp\lp1-P_k\lp b_n\rp\rp$, \hyperref[pnbound]{Theorem~\ref*{pnbound}} implies that $\lp a_n\rp$ satisfies the hypotheses of \hyperref[comparativerecurrence]{Lemma~\ref*{comparativerecurrence}}, which we apply to conclude the proof.
\end{proof}

\section{Applying the Hilbert Irreducibility and the Frobenius Density Theorems}
\label{puttingitalltogether}

In this section, for any polynomial $f(c,x)\in\Q[c][x]$ and any $a\in\Q$, we will write $f_a$ for the specialization of $f$ at $c=a$; that is, $f_a=f_a(x)=f(a,x)\in\Q[x]$.
Below is a version of the Hilbert Irreducibility Theorem, one which we will apply in the proof of \hyperref[final]{Theorem~\ref*{final}}.

\begin{named}{Hilbert Irreducibility Theorem}\label{HIT}
Let $f(c,x)\in\Z[c][x]$, let $K$ be the splitting field of $f(c,x)$ over $\Q(c)$, and for any $a\in\Z$, let $K_a$ be the splitting field of $f_a$ over $\Q$.
Suppose that $f(c,x)$ has no repeated roots \textup{(}in $\overline{\Q(c)}$\textup{)}. Then there exists a ``thin set'' $A\subset\Z$ such that for all $a\in\Z\setminus A$,
\[
f_a\text{ has no repeated roots}
\hspace{10px}\text{and}\hspace{10px}
\Gal{\lp K_a/\Q\rp}\simeq\Gal{\lp K/\Q(c)\rp}.
\]
\end{named}
\begin{HITremark}\label{HITremark}
The Hilbert Irreducibility Theorem is normally stated for irreducible polynomials (as in~\cite{Serre}).
To obtain the version stated above, let $g(c,x)$ be the minimal polynomial of a primitive element of $K/\Q(c)$, which is irreducible over $\Q(c)$.
Then specialize $g(c,x)$ instead of $f(c,x)$.
Moreover, if $f(c,x)$ has no repeated roots in $\overline{\Q(c)}$, then there are only finitely many $a\in\Q$ for which $f_a(x)$ has a repeated root in $\overline{\Q}$ (these are precisely the $a$ for which the discriminant $\Disc{f(c,x)}$ vanishes under the specialization at $c=a$). For more on the connection between the Hilbert Irreducibility Theorem and Galois theory, see, for example,~\cite{CohenHIT},~\cite[Chapter VIII]{LangDiophantineGeometry}, and~\cite[Chapter 1]{Volklein}.

As for the size of the ``thin set'' $A$, we know there is some constant $C$ such that for all $X\in\Z_{>0}$,
\[
\lv\lb a\in A\mid a\leq X\rb\rv\leq C\sqrt{X}
\]
(See~\cite{Serre}, Section~9.7, for more details).
In particular, there are infinitely many integers for which the conclusion of the theorem is true.
\end{HITremark}

Next, we recall a case of the Frobenius Density Theorem.
(See~\cite{SL} for more details.)

\begin{named}{Frobenius Density Theorem}\label{FDT}
Suppose that $f(x)\in\Z[x]$ is a monic polynomial with no repeated roots.
Let $G=\Gal{(f/\Q)}$ and $P\subseteq\mathcal{P}$ be the set of primes $p$ such that $[f]_p$ has a root in $\F_p$. Then
\[
\mu(P)
=\frac{1}{\lv G\rv}\cdot
\lv\lb\sigma\in G
\mid
\sigma\text{ fixes a root of $f$}
\rb\rv.
\]
\end{named}
\begin{FDTremark}\label{FDTremark}
If $f,g\in\Z[x]$ satisfy the hypotheses of the \ref{FDT}, the sets
\[
P_f=\{p\in\mathcal{P}\mid[f]_p\text{ has a root in }\F_p\}
\hspace{10px}\text{and}\hspace{10px}
P_g=\{p\in\mathcal{P}\mid[g]_p\text{ has a root in }\F_p\}
\]
are probabilistically independent (in the sense that $\mu\lp P_f\cap P_g\rp=\mu\lp P_f\rp\cdot\mu\lp P_g\rp$) if and only if the splitting fields of $f$ and $g$ are linearly disjoint over $\Q$. This follows immediately from the fact that the Galois group of a compositum of fields is the direct product of the Galois groups if and only if the fields are linearly disjoint.
\end{FDTremark}

In light of the \ref{FDT}, one might hope that given $f\in\Z[x]$ and $p\in\mathcal{P}$, the roots of $\lc\Phi_{f,n}\rc_p$ are precisely the points of $\lp\F_p,[f]_p\rp$ of period $n$, but---as mentioned in \hyperref[intro]{Section~\ref*{intro}}---this hope would be in vain.
Indeed, even before reducing mod $p$, if $\alpha\in\overline{\Q}$ is a point of period $n$ in $\lp\overline{\Q},f\rp$, then $\Phi_{f,n}(\alpha)=0$, but the converse is not always true---that is, there are examples of $(K,f)$, $n$, $\alpha$, and $d$, where $d<n$, $\alpha$ is a point of period $d$, but $\Phi_{f,n}(\alpha)=0$, see \cite[Example 4.2]{SilvermanADS}.
In general, if $\alpha\in\overline{\Q}$ and $\Phi_{f,n}(\alpha)=0$, then $\alpha$ is of period $d$ for some $d\leq n$, and $d<n$ is possible only if the polynomial derivative of $f^d$ evaluated at $\alpha$ is a root of unity; this quantity is known as the \emph{multiplier of $\alpha$}.
The way in which the period depends on the multiplier is the content of the following theorem~\cite[Theorem 4.5]{SilvermanADS}.
\begin{named}{Roots and Multipliers Theorem}\label{mrpe}
Suppose that $K$ is a field, $f\in K[x]$, $n\in\Z_{>0}$, and $\alpha\in\overline{K}$ satisfies $\Phi_{f,n}(\alpha)=0$.
Let $\lambda=(f^m)'(\alpha)$ where $m$ is the (least) period of $\alpha$. Then either
\begin{enumerate}
\item $n=m$,
\item $n=mj$, when $\lambda$ is a primitive $j$th root of unity, or
\item $n=mjp^e$, with $e\in\Z_{>0}$, when $\lambda$ is a primitive $j$th root of unity and $\Char{K}=p>0$.
\end{enumerate}
Conversely, if $\alpha\in\overline{K}$ has period $n$ in $\lp K(\alpha),f\rp$, then $\Phi_{f,n}(\alpha)=0$.
\end{named}

Luckily, given $f\in\Z[x]$ and $n\in\Z_{>0}$, the following corollary provides a sufficient condition that ensures that for all but finitely many primes $p\in\mathcal{P}$, the dynamical systen $\lp\F_p,[f]_p\rp$ has a point of period $n$ if and only if $\lc\Phi_{f,n}\rc_p$ has a root.
In the proof of \hyperref[final]{Theorem~\ref*{final}}, we will use the work in \hyperref[galois]{Section~\ref*{galois}} to ensure that the polynomials obtained by applying the \ref{HIT} satisfy this sufficient condition.


\begin{rootsandcycles}\label{rootsandcycles}
Let $f\in\Z[x]$ and $n\in\Z_{>0}$, and suppose that $f^n(x)-x$ has no repeated roots.
Then for all but finitely many $p\in\mathcal{P}$,
\[
\lc\Phi_{f,n}\rc_p\text{ has a root in }\F_p
\text{ if and only if }
\lp\F_p,[f]_p\rp\text{ has a point of period }n.
\]
\end{rootsandcycles}
\begin{proof}
As pointed out in Section~2 of~\cite{VH}, for example, if $f^n(x)-x$ has no repeated roots, then for all $\alpha\in\overline{\Q}$,
\[
\Phi_{f,n}(\alpha)=0
\text{ if and only if }
\alpha\text{ is a point of period $n$ in }\lp \Q(\alpha),f\rp.
\]
As usual, for any $\alpha\in\overline{\Q}$ and $p\in\mathcal{P}$, we say $p$ divides $\alpha$ if there exists a number field $K$ containing $\alpha$ and a prime ideal $\mathfrak{p}\subseteq\mathcal{O}_K$ such that $\mathfrak{p}\mid(p)$ and $\ord_{\mathfrak{p}}(\alpha)>0$.

Let $p\in\mathcal{P}$, and suppose that $\lc\Phi_{f,n}\rc_p$ has a root.
Let $K$ be the splitting field of $\Phi_{f,n}$, choose any prime $\mathfrak{p}$ lying over $p$, and denote by $\overline{\hspace{-4px}\phantom{+}\cdot\phantom{+}\hspace{-4px}}$ the reduction $\mathcal{O}_K\to\mathcal{O}_K/\mathfrak{p}$.
Since $\mathcal{O}_K/\mathfrak{p}$ is an extension of $\F_p$ and $\lc\Phi_{f,n}\rc_p$ has a root, there exists $a\in\Z\subseteq\mathcal{O}_K$ such that $\lc\Phi_{f,n}\rc_p\lp\overline{a}\rp=0$.
Since $\Phi_{f,n}$ splits in $K$, we know that $\lc\Phi_{f,n}\rc_p$ splits in $\mathcal{O}_K/\mathfrak{p}$ and the roots of $\Phi_{f,n}$ map onto the roots of $\lc\Phi_{f,n}\rc_p$ under $\overline{\hspace{-4px}\phantom{+}\cdot\phantom{+}\hspace{-4px}}$; choose any $\alpha\in K$ such that $\Phi_{f,n}(\alpha)=0$ and $\overline{\alpha}=\overline{a}$.

Now, by the \ref{mrpe}, we know that $\overline{a}$ is a periodic point of $\lp\F_p,[f]_p\rp$ of period at most $n$; let's suppose its period is strictly less than $n$ (so in particular, $n>1$).
This implies that there exists $j\in\lb1,\ldots,n-1\rb$ such that
\[
\overline{f^j(\alpha)-\alpha}=\lp[f]_p\rp^j\lp\overline{a}\rp-\overline{a}=0;
\]
that is, $\mathfrak{p}$ divides $f^j(\alpha)-\alpha$.
We know that $\alpha$ has period $n$ in $\lp K,f\rp$, so the points $\alpha,f(\alpha),f^2(\alpha),\dots,f^{n-1}(\alpha)$ are pairwise distinct; thus, there are only finitely many prime ideals of $\mathcal{O}_K$ dividing their differences, as desired.

Now suppose that $[f]_p$ has a point of period $n$ in $\lp\F_p,[f]_p\rp$.
It is easy to see that $[f^n]_p=\lp[f]_p\rp^n$, so $\lc\Phi_{f,n}\rc_p= \Phi_{[f]_p,n}$.
By the \ref{mrpe}, with $K=\F_p$, we know that $\lc\Phi_{f,n}\rc_p$ has a root in $\F_p$.
\end{proof}

Finally, we can apply \hyperref[recurrence]{Corollary~\ref*{recurrence}}, \hyperref[QbarQ]{Corollary~\ref*{QbarQ}}, and the results mentioned above to prove \hyperref[final]{Theorem~\ref*{final}}.

\begin{proof}[Proof of \cref{final}]
Let $\mathcal{T}=\lb J\subseteq\lb1,\ldots,M\rb\,\big\vert\,\lv J\rv=2\rb$, set $t=\lv\mathcal{T}\rv=\binom{M}{2}$, and choose any bijection $\beta:\mathcal{T}\to\lb1,\ldots,t\rb$. 
For any $J\in\mathcal{T}$, let $A_J$ denote the arithmetic progression $\lp\beta(J),\beta(J)+t,\beta(J)+2t,\dots\rp$.
Define $s_{A_J,0}=0$ and
\[
s_{A_J,i}
=s_{A_J,i-1}+2\lp1-s_{A_J,i-1}\rp P_{k}(\beta(J)+t(i-1))\lp1-P_k(\beta(J)+t(i-1))\rp
\hspace{10px}\text{for all}\hspace{10px}
i\in\Z_{>0}.
\]
By \hyperref[recurrence]{Corollary~\ref*{recurrence}}, for each choice of $J$ we know $s_{A_J,i}\to 1$ as $i\to\infty$.
Thus, there exists $N_0\in\Z_{>0}$ such that $s_{A_J,N_0}>1-\frac{\epsilon}{t}$ for all $J\in\mathcal{T}$.
Let $N=\lp tN_0\rp!$ and set $f=x^k+c\in\Q(c)[x]$, so that for any $a\in\Z$, $f_a=x^k+a\in\Z[x]$.
Next, let
\[
\mathcal{F}\lp a\rp=\lb\Sigma_{f_a,d}\mid d\in\Z_{>0}\text{ and }d\mid N\rb,
\]
where $\Sigma_{f_a,d}$ is the splitting field of $\Phi_{f_a,d}$ over $\Q$.
By Theorem~D of~\cite{MorGalGroups}, we know that $f^N(x)-x$ has no repeated roots in $\overline{\Q(c)}$, so by \hyperref[automorphismlemma]{Lemma~\ref*{automorphismlemma}}, for any $m\in\Z$, we see that $(f+m)^N(x)-x$ has no repeated roots either.
Thus, by \hyperref[QbarQ]{Corollary~\ref*{QbarQ}} and the \ref{HIT}, there exist infinitely many $M$-tuples $( m_1,\ldots,m_M)\in\Z^M$ such that
\begin{itemize}
\item
any field in $\bigcup_{j=1}^M{\mathcal{F}\lp m_j\rp}$ is linearly disjoint from the compositum of the others,
\item
if $\Sigma_{f_{m_j},d}\in\bigcup_{j=1}^M{\mathcal{F}\lp m_j\rp}$, then $\Gal{\lp\Sigma_{f_{m_j},d}/\Q\rp}\simeq\lp\Z/d\Z\wr S_{r_k(d)}\rp$, and
\item
for any $j\in\lb 1,\ldots,M\rb$, we know $\lp f_{m_j}\rp^N(x)-x$ has no repeated roots.
\end{itemize}
We will prove that for any such $\lp m_1,\ldots,m_M\rp$,
\[
\mu\lp\lb p\in\mathcal{P}\mid
\lb x^k+m_1,\ldots,x^k+m_M\rb\text{ is dynamically distinguishable mod }p\rb\rp
>1-\epsilon.
\]

To begin, fix such an $M$-tuple $\lp m_1,\ldots,m_M\rp$.
We introduce a bit of simplifying notation.
For any $J=\lb j,j^\prime\rb\in\mathcal{T}$, we will compare those $\Phi_{x^k+m_j,i},\Phi_{x^k+m_{j^\prime},i}$ for $i$ in the truncated arithmetic progression $\lp\beta(J)+t(i-1)\mid i\in\lb1,\ldots,N_0\rb\rp$.
To make this analysis more convenient, for any $J\in\mathcal{T}$ and $j\in J$, write
\[
\Phi_{J,j,i}=\Phi_{x^k+m_j,\beta(J)+t(i-1)}.
\]
Using this notation, we can define for any $J\in\mathcal{T}$:
\[
\Phi_J=\prod_{j\in J}{\prod_{i=1}^{N_0}{\Phi_{J,j,i}}}
\hspace{20px}\text{and}\hspace{20px}
G_J=\Gal{\lp\Phi_J/\Q\rp}.
\]
Now set 
\[
\Phi=\prod_{J\in\mathcal{T}}{\Phi_J}
\hspace{20px}\text{and}\hspace{20px}
G=\Gal{\lp\Phi/\Q\rp}.
\]
Note that 
\[
\Phi\text{ divides } \prod_{j=1}^M \lp \lp f_{m_j}\rp^N(x)-x\rp \text{ in }\Q[x]
\]
by the definition of the dynatomic polynomials, so 

\[
\Phi\text{ has distinct roots in }\overline{\Q}
\hspace{20px}\text{and}\hspace{20px}
G\simeq\prod_{J\in\mathcal{T}}{G_J}
\]
by our choice of $\lp m_1,\ldots,m_M\rp\in\Z^M$.

Next, we introduce the sets of primes whose natural densities we will compute; namely, for any $J\in\mathcal{T}$ and $i\in\lb1,\ldots,N_0\rb$, define
\begin{align*}
P^\Gamma_{J,i}&=\lb p\in\mathcal{P}\mid
\text{exactly one of }\lb\Gamma_{x^k+m_j,p}\mid j\in J\rb\text{ has a $(\beta(J)+t(i-1))$-cycle}\rb,\\
P^\Phi_{J,i}&=\lb p\in\mathcal{P}\mid
\text{exactly one of }\lb\lc\Phi_{J,j,i}\rc_p\mid j\in J\rb\text{ has a root in }\F_p\rb.
\end{align*}
As we will compare these sets to proportions of Galois groups, we define for any $J\in\mathcal{T}$,
\[
C_J=\lb\sigma\in G_J\mid\text{for some }i\in\lb1,\ldots,N_0\rb\text{, }\sigma\text{ fixes a root of exactly one of }\lb\Phi_{J,j,i}\mid j\in J\rb\rb.
\]
Now, set $P^\Phi_J=\bigcup_{i=1}^{N_0}{P^\Phi_{J,i}}$ and apply the \ref{FDT} to $\Phi_J$ to see that
\[
\mu\lp P^\Phi_J\rp=\frac{\lv C_J\rv}{\lv G_J\rv}.
\]
Next, recall that \hyperref[rootsandcycles]{Corollary~\ref*{rootsandcycles}} implies that for any $J\in\mathcal{T}$ and $i\in\lb1,\ldots,N_0\rb$, the symmetric difference of $P^\Gamma_{J,i}$ and $P^\Phi_{J,i}$ is finite.
Thus,
\begin{align*}
\mu\lp\lb p\in\mathcal{P}\mid
\lb x^k+m_1,\ldots,x^k+m_M\rb\text{ is dynamically distinguishable mod }p\rb\rp\hspace{-290px}\\
&=\mu\lp\bigcap_{J\in\mathcal{T}}
{\lb p\in\mathcal{P}\mid \lb x^k+m_j\mid j\in J\rb\text{ is dynamically distinguishable mod }p\rb}\rp\\
&\geq\mu\lp\bigcap_{J\in\mathcal{T}}{\lp\bigcup_{i=1}^{N_0}{P^\Gamma_{J,i}}\rp}\rp\\
&=\mu\lp\bigcap_{J\in\mathcal{T}}{\lp\bigcup_{i=1}^{N_0}{P^\Phi_{J,i}}\rp}\rp\\
&=\mu\lp\bigcap_{J\in\mathcal{T}}{P^\Phi_J}\rp\\
&=\prod_{J\in\mathcal{T}}{\frac{\lv C_J\rv}{\lv G_J\rv}},
\end{align*}
where the last step follows from \hyperref[FDTremark]{Remark~\ref*{FDTremark}}.

We will conclude the proof by showing that if $J\in\mathcal{T}$, then $\frac{\lv C_J\rv}{\lv G_J\rv}>1-\frac{\epsilon}{t}$, whence
\[
\prod_{J\in\mathcal{T}}{\frac{\lv C_J\rv}{\lv G_J\rv}}>\lp 1-\frac{\epsilon}{t}\rp^t>1-\epsilon.
\]
By \hyperref[GaloisWreathActions]{Remark~\ref*{GaloisWreathActions}}, \hyperref[recurrence]{Corollary~\ref*{recurrence}}, and our choice of $\lp m_1,\ldots,m_M\rp$, we know that 
\[
\frac{\lv C_J\rv}{\lv G_J\rv}=s_{A_J,N_0},
\]
so we are done by
our original choice of $N_0$.

\end{proof}

\section*{Acknowledgements} \label{Acknowledgements}

We would like to thank the referee for a careful reading of the paper and very helpful comments on the presentation. We also thank Patrick Morton for helpful comments about the proofs in his series of papers on dynatomic polynomials, Rafe Jones for pointing us to Morton's work, and Robert Lemke Oliver for many constructive conversations regarding the topics in this paper.

\bibliography{DistinguishingPolynomials}

\newcommand{\etalchar}[1]{$^{#1}$}
\providecommand{\bysame}{\leavevmode\hbox to3em{\hrulefill}\thinspace}
\providecommand{\MR}{\relax\ifhmode\unskip\space\fi MR }
\providecommand{\MRhref}[2]{%
  \href{http://www.ams.org/mathscinet-getitem?mr=#1}{#2}
}
\providecommand{\href}[2]{#2}
\begin{thebibliography}{BGTW16}

\bibitem[Bac91]{Bach}
Eric Bach, \emph{\href{http://dx.doi.org/10.1016/0890-5401(91)90001-I}{Toward a
  theory of {P}ollard's rho method}}, Inform. and Comput. \textbf{90} (1991),
  no.~2, 139--155. \MR{1094034}

\bibitem[BB13]{BachBridy}
Eric Bach and Andrew Bridy,
  \emph{\href{http://dx.doi.org/10.1016/j.laa.2013.04.014}{On the number of
  distinct functional graphs of affine-linear transformations over finite
  fields}}, Linear Algebra Appl. \textbf{439} (2013), no.~5, 1312--1320.
  \MR{3067805}

\bibitem[BCE15]{BCE}
J{\'e}r{\'e}my Blanc, Jung~Kyu Canci, and Noam~D. Elkies,
  \emph{\href{http://dx.doi.org/10.1093/imrn/rnv063}{Moduli spaces of quadratic
  rational maps with a marked periodic point of small order}}, Int. Math. Res.
  Not. IMRN (2015), no.~23, 12459--12489. \MR{3431627}

\bibitem[BGTW16]{BGTW}
Elisa Bellah, Derek Garton, Erin Tannenbaum, and Noah Walton,
  \emph{\href{http://arxiv.org/abs/1609.07667}{A probabilistic heuristic for
  counting components of functional graphs of polynomials over finite fields}},
  ArXiv e-prints (2016), to appear in
  \emph{\href{http://msp.org/scripts/coming.php?jpath=involve}{Involve}}.

\bibitem[Bra89]{Branner}
Bodil Branner, \emph{\href{http://dx.doi.org/10.1090/psapm/039/1010237}{The
  {M}andelbrot set}}, Chaos and fractals ({P}rovidence, {RI}, 1988), Proc.
  Sympos. Appl. Math., vol.~39, Amer. Math. Soc., Providence, RI, 1989,
  pp.~75--105. \MR{1010237}

\bibitem[BS15]{BS}
C.~Burnette and E.~Schmutz,
  \emph{\href{http://arxiv.org/abs/1508.04193}{Periods of Iterated Rational
  Functions over a Finite Field}}, ArXiv e-prints (2015), to appear in the
  \emph{\href{http://dx.doi.org/10.1142/S1793042117500713}{International
  Journal of Number Theory}}.

\bibitem[Coh81]{CohenHIT}
S.~D. Cohen, \emph{\href{http://dx.doi.org/10.1112/plms/s3-43.2.227}{The
  distribution of {G}alois groups and {H}ilbert's irreducibility theorem}},
  Proc. London Math. Soc. (3) \textbf{43} (1981), no.~2, 227--250. \MR{628276}

\bibitem[DeM07]{DeMarco}
Laura DeMarco, \emph{\href{http://dx.doi.org/10.1090/S0894-0347-06-00527-3}{The
  moduli space of quadratic rational maps}}, J. Amer. Math. Soc. \textbf{20}
  (2007), no.~2, 321--355. \MR{2276773}

\bibitem[Fat19]{Fatou1}
P.~Fatou, \emph{\href{http://www.numdam.org/item?id=BSMF_1919__47__161_0}{Sur
  les \'equations fonctionnelles}}, Bull. Soc. Math. France \textbf{47} (1919),
  161--271. \MR{1504787}

\bibitem[Fat20]{Fatou2}
\bysame, \emph{\href{http://www.numdam.org/item?id=BSMF_1920__48__208_1}{Sur
  les \'equations fonctionnelles}}, Bull. Soc. Math. France \textbf{48} (1920),
  208--314. \MR{1504797}

\bibitem[FG14]{FG}
Ryan Flynn and Derek Garton,
  \emph{\href{http://dx.doi.org/10.1142/S1793042113501224}{Graph components and
  dynamics over finite fields}}, Int. J. Number Theory \textbf{10} (2014),
  no.~3, 779--792. \MR{3190008}

\bibitem[Isa08]{IsaacsGroupTheory}
I.~Martin Isaacs, \emph{\href{http://dx.doi.org/10.1090/gsm/092}{Finite group
  theory}}, Graduate Studies in Mathematics, vol.~92, American Mathematical
  Society, Providence, RI, 2008. \MR{2426855}

\bibitem[Jul18]{Julia}
Gaston Julia, \emph{\href{http://eudml.org/doc/234994}{M\'emoire sur
  l’iteration des fonctions rationnelles}}, Journal de Math. Pures et Appl.
  \textbf{8} (1918), 47–--245.

\bibitem[KLM{\etalchar{+}}16]{Ketal}
Sergei~V. Konyagin, Florian Luca, Bernard Mans, Luke Mathieson, Min Sha, and
  Igor~E. Shparlinski,
  \emph{\href{http://dx.doi.org/10.1016/j.jctb.2015.07.003}{Functional graphs
  of polynomials over finite fields}}, J. Combin. Theory Ser. B \textbf{116}
  (2016), 87--122. \MR{3425238}

\bibitem[Lan83]{LangDiophantineGeometry}
Serge Lang,
  \emph{\href{http://dx.doi.org/10.1007/978-1-4757-1810-2}{Fundamentals of
  {D}iophantine geometry}}, Springer-Verlag, New York, 1983. \MR{715605}

\bibitem[Lan02]{Lang}
\bysame, \emph{\href{http://dx.doi.org/10.1007/978-1-4613-0041-0}{Algebra}},
  third ed., Graduate Texts in Mathematics, vol. 211, Springer-Verlag, New
  York, 2002. \MR{1878556}

\bibitem[Lev11]{Levy}
Alon Levy, \emph{\href{http://dx.doi.org/10.4064/aa146-1-2}{The space of
  morphisms on projective space}}, Acta Arith. \textbf{146} (2011), no.~1,
  13--31. \MR{2741188}

\bibitem[Mor96]{MorCurves}
Patrick Morton,
  \emph{\href{http://www.numdam.org/item?id=CM_1996__103_3_319_0}{On certain
  algebraic curves related to polynomial maps}}, Compositio Math. \textbf{103}
  (1996), no.~3, 319--350. \MR{1414593 (97m:14030)}

\bibitem[Mor98]{MorGalGroups}
\bysame, \emph{\href{http://dx.doi.org/10.1006/jabr.1997.7304}{Galois groups of
  periodic points}}, J. Algebra \textbf{201} (1998), no.~2, 401--428.
  \MR{1612390}

\bibitem[Mor11]{MorCorrection}
\bysame, \emph{\href{http://doi.org/10.1112/S0010437X1000480X}{Corrigendum:
  `{O}n certain algebraic curves related to polynomial maps, {C}ompositio
  {M}ath. 103 (1996), 319--350'}}, Compos. Math. \textbf{147} (2011), no.~1,
  332--334. \MR{2771135}

\bibitem[MP94]{MP}
Patrick Morton and Pratiksha Patel,
  \emph{\href{http://dx.doi.org/10.1112/plms/s3-68.2.225}{The {G}alois theory
  of periodic points of polynomial maps}}, Proc. London Math. Soc. (3)
  \textbf{68} (1994), no.~2, 225--263. \MR{1253503}

\bibitem[OS16]{OstafeSha}
Aline Ostafe and Min Sha,
  \emph{\href{http://dx.doi.org/10.1090/conm/669}{Counting dynamical systems
  over finite fields}}, Dynamics and Numbers, Contemp. Math., vol. 669, Amer.
  Math. Soc., Providence, RI, 2016, pp.~187--204.

\bibitem[Pol75]{Pollard}
J.~M. Pollard, \emph{\href{http://dx.doi.org/10.1007/BF01933667}{A {M}onte
  {C}arlo method for factorization}}, Nordisk Tidskr. Informationsbehandling
  (BIT) \textbf{15} (1975), no.~3, 331--334. \MR{0392798}

\bibitem[Sch94]{Schleicher}
Dierk Schleicher,
  \emph{\href{http://search.proquest.com/docview/304125447}{Internal addresses
  in the {M}andelbrot set and irreducibility of polynomials}}, ProQuest LLC,
  Ann Arbor, MI, 1994, Thesis (Ph.D.)--Cornell University. \MR{2691195}

\bibitem[Ser97]{Serre}
Jean-Pierre Serre,
  \emph{\href{http://dx.doi.org/10.1007/978-3-663-10632-6}{Lectures on the
  {M}ordell-{W}eil theorem}}, third ed., Aspects of Mathematics, Friedr. Vieweg
  \& Sohn, Braunschweig, 1997, Translated from the French and edited by Martin
  Brown from notes by Michel Waldschmidt, With a foreword by Brown and Serre.
  \MR{1757192}

\bibitem[Sil98]{SilvermanModuli}
Joseph~H. Silverman,
  \emph{\href{http://dx.doi.org/10.1215/S0012-7094-98-09404-2}{The space of
  rational maps on {$\bold P^1$}}}, Duke Math. J. \textbf{94} (1998), no.~1,
  41--77. \MR{1635900}

\bibitem[Sil07]{SilvermanADS}
\bysame, \emph{\href{http://dx.doi.org/10.1007/978-0-387-69904-2}{The
  arithmetic of dynamical systems}}, Graduate Texts in Mathematics, vol. 241,
  Springer, New York, 2007. \MR{2316407}

\bibitem[SL96]{SL}
P.~Stevenhagen and H.~W. Lenstra, Jr.,
  \emph{\href{http://websites.math.leidenuniv.nl/algebra/chebotarev.pdf}{Chebotar\"ev
  and his density theorem}}, Math. Intelligencer \textbf{18} (1996), no.~2,
  26--37. \MR{1395088}

\bibitem[VH92]{VH}
Franco Vivaldi and Spyros Hatjispyros,
  \emph{\href{http://stacks.iop.org/0951-7715/5/961}{Galois theory of periodic
  orbits of rational maps}}, Nonlinearity \textbf{5} (1992), no.~4, 961--978.
  \MR{1174226}

\bibitem[V{\"o}l96]{Volklein}
Helmut V{\"o}lklein,
  \emph{\href{http://dx.doi.org/10.1017/CBO9780511471117}{Groups as {G}alois
  groups}}, Cambridge Studies in Advanced Mathematics, vol.~53, Cambridge
  University Press, Cambridge, 1996, An introduction. \MR{1405612}

\end{thebibliography}
\bibliographystyle{amsalpha}

\end{document}